\documentclass{amsart}
\usepackage{amssymb, amsmath, latexsym, mathabx, dsfont}
\usepackage{makecell,array,enumitem,multirow}
\usepackage{stmaryrd}

\renewcommand{\baselinestretch}{\baselinestretch}
\renewcommand{\baselinestretch}{1.1}
\numberwithin{equation}{section}

\newtheorem{thm}{Theorem}[section]

\newtheorem{cor}[thm]{Corollary}

\theoremstyle{definition}
\newtheorem{rmk}[thm]{Remark}
\theoremstyle{remark}

\newcommand{\ra}{{\ \rightarrow\ }}

\newcommand{\gen}{\text{gen}}

\newcommand{\z}{{\mathbb Z}}

\newcommand{\n}{{\mathbb N}}

\newcommand{\Mod}[1]{\ (\mathrm{mod}\ #1)}

\newcommand{\df}[1]{\langle #1 \rangle}

\begin{document}

\title{Prime-universal diagonal quadratic forms}

\author[Jangwon Ju and et tal] {Jangwon Ju, Daejun Kim, Kyoungmin Kim, Mingyu Kim, and Byeong-Kweon Oh}

\address{Department of Mathematics, University of Ulsan, Ulsan, 44610, Republic of Korea}
\email{jangwonju@ulsan.ac.kr}
\thanks{This work of the first author was supported by the National Research Foundation of Korea(NRF) grant funded by the Korea government(MSIT) (NRF-2019R1F1A1064037).}

\address{Research Institute of Mathematics, Seoul National University, Seoul 08826, Korea}
\email{goodkdj@snu.ac.kr}

\address{Department of Mathematics, Sungkyunkwan University, Suwon 16419, korea}
\email{kiny30@skku.edu}
\thanks{This work of the third author was supported by the National Research Foundation of Korea(NRF) grant funded by the Korea government(MSIT) (NRF-2016R1A5A1008055  and NRF-2018R1C1B6007778)}

\address{Department of Mathematics, Sungkyunkwan University, Suwon 16419, korea}
\email{kmg2562@skku.edu}
\thanks{This work of the fourth author was supported by the National Research Foundation of Korea (NRF-2019R1A6A3A01096245).}

\address{Department of Mathematical Sciences and Research Institute of Mathematics, Seoul National University, Seoul 08826, Korea}
\email{bkoh@snu.ac.kr}
\thanks{This work of the second and fifth authors were supported by the National Research Foundation of Korea (NRF-2019R1A2C1086347).}



\begin{abstract}  A (positive definite and integral) quadratic form  is said to be {\it prime-universal} if it represents all primes.
Recently, Doyle and Williams in \cite{dw} classified all prime-universal  diagonal  ternary quadratic forms, and all prime-universal  diagonal  quaternary quadratic forms under two conjectures proposed by themselves.  
In this article, we classify all  prime-universal diagonal quadratic forms regardless of ranks. Furthermore, we prove, so called, $67$-Theorem for a diagonal quadratic form to be prime-universal.      
\end{abstract}

\maketitle

\section{Introduction}

A positive definite integral quadratic form 
$$
f(x_1,x_2,\dots,x_n)=\sum_{i,j=1}^n f_{ij}x_ix_j \quad (f_{ij}=f_{ji} \in \z)
$$
is said to be {\it prime-universal} if for any prime $p$, the diophantine equation 
$$
f(x_1,x_2,\dots, x_n)=p
$$ 
has an integer solution $(x_1,x_2,\dots,x_n) \in \z^n$.  There are many examples of prime-universal quadratic forms. For example, any quadratic form that represents all positive integers, which we call  a {\it  universal} quadratic form, is also prime-universal.  Hence by 15-Theorem  given by Conway and Schneeberger (for this, see \cite{b}), any quadratic form that represents all integers less than or equal to $15$ is universal, which is also prime-universal.  On the other hand, there are many prime-universal quadratic forms that are not universal.

Recently,  Doyle and Williams \cite{dw}  classified all  prime-universal diagonal ternary quadratic forms, which are, in fact, not universal. Furthermore, they classified all prime-universal diagonal quaternary quadratic forms under the two conjectures proposed by themselves.    

In this article, we classify all  prime-universal diagonal  quadratic forms without any assumption. Furthermore, we prove that  a diagonal quadratic form is prime-universal if and only if it represents 
$$
  2, \ 3, \ 5, \ 7, \ 13, \ 17, \ 23, \ 41, \ 43, \ \text{and} \ \ 67.
$$

Let $f=\sum f_{ij} x_ix_j$ be a (positive definite integral) quadratic form.  The symmetric matrix corresponding to $f$ is defined by $M_f=(f_{ij})$. If $f$ is diagonal, then we simply write  
$$
f=\df{f_{11},f_{22},\dots,f_{nn}}.
$$ 
Note that the symmetric matrix corresponding to a diagonal form is a diagonal matrix.  For a positive integer $a$,  we say $a$ is represented by $f$, denoted by $a \ra f$, if the diophantine equation $f(x_1,\dots,x_n)=a$ has an integer solution $(x_1,\dots,x_n) \in \z^n$.  For a prime $p$, if it has a solution $(x_1,\dots,x_n) \in \z_p$, where $\z_p$ is the $p$-adic integer ring, then we say $a$ is represented by $f$ over $\z_p$. Furthermore, if $a$ is represented by $f$ over $\z_p$  for any prime $p$, then we say $a$ is {\it locally} represented by $f$. 
The set of all integers that are represented by $f$ will be denoted by $Q(f)$. 
The genus of $f$, denoted by $\gen(f)$, is the set of quadratic forms that are isometric to $f$ over  $\z_p$ for any prime $p$. It is well known that an integer $a$ is locally represented by $f$ if and only if there is a quadratic form $g \in \gen(f)$ that represents $a$. 
The set of integers that are locally represented by $f$ will be denoted by $Q(\gen(f))$. 

Any unexplained notations and terminologies can be found in \cite{ki} or \cite{om}.

\section{Prime-universal diagonal quaternary quadratic forms}
In this section, we classify all prime-universal diagonal quaternary quadratic forms.   In \cite{dw}, the authors classified all prime-universal diagonal quaternary quadratic forms except for the following $27$ forms:
\begin{equation}\label{candi}
\begin{array}{llll}
\df{2,3,4,5}, &\df{2,3,4,11}, & \df{2,3,5,5}, & \df{2,3,5,11}, \\
\df{2,3,5,13},&\df{2,3,5,14}, & \df{2,3,5,16},& \df{2,3,5,17}, \\
\df{2,3,5,h} & \multicolumn{3}{l}{\text{for $h=20,\dots,27,29,\dots,33,35,36,38,40,41,43$.}}
\end{array}
\end{equation}
Hence it suffices to prove the prime-universalities  of these  $27$ diagonal forms.  

In \cite{poly}, \cite{regular}, and \cite{pentagonal}, we developed a method that determines whether or not integers in an arithmetic progression are represented by some particular ternary quadratic form. 
We briefly introduce this method for those who are unfamiliar with it. 
Let $d$ be a positive integer and let $a$ be a nonnegative integer $(a\leq d)$. 
We define 
$$
S_{d,a}=\{dn+a \mid n \in \mathbb N \cup \{0\}\}.
$$
For two quadratic forms $f$ and $g$, we define
$$
R(g,d,a)=\{v \in (\mathbb{Z}/d\mathbb{Z})^3 \mid vM_gv^t\equiv a \ (\text{mod }d) \}
$$
and
$$
R(f,g,d)=\{T\in M_3(\mathbb{Z}) \mid  T^tM_fT=d^2M_g \}.
$$
A coset (or, a vector in the coset) $v \in R(g,d,a)$ is said to be {\it good} with respect to $f,g,d, \text{ and }a$ if there is a $T\in R(f,g,d)$ such that $\frac1d \cdot vT^t \in \mathbb{Z}^3$.  
The set of all good vectors in $R(g,d,a)$ is denoted by $R_f(g,d,a)$.   
If  $R(g,d,a)=R_f(g,d,a)$, we write  
$$
g\prec_{d,a} f.
$$ 
If $g\prec_{d,a} f$, then by Lemma 2.2 of \cite{regular},  we have 
\begin{equation}\label{good}
S_{d,a}\cap Q(g) \subset Q(f).
\end{equation}
In particular, if $g\prec_{d,a} f$ for any $g \in \gen(f)$, then we have $Q(\gen(f)) \cap S_{d,a} \subset Q(f)$. 

 A MAPLE based  computer program for computing the above sets  is available upon request to the authors.

\begin{thm} \label{start} Both quaternary quadratic forms $f=\langle 2,3,4,5\rangle$ and $\langle2,3,4,11\rangle$  are prime-universal. 
\end{thm}

\begin{proof}   Since the proofs are quite similar to each other, we only provide the proof of the former case.

Note that the class number of $g=\langle 2,3,4\rangle$ is $2$, and its genus mate is $\df{1,2,12}$. If an even integer is represented by $\df{1,2,12}$, then it is also represented by $\df{4,2,12}$, which is a subform of $g$. Therefore $Q(\gen(g)) \cap 2\z \subset Q(g)$, and consequently  $g$    
represents all nonnegative even integers except for those integers of the form $2^{2s+1}(8t+5)$ for some nonnegative integers $s$ and $t$.  

Since one may directly check that $p \ra f$ for any prime $p$ such that $p\le 113$, we assume that $p\ge 127$.
Then, one may easily check that
$$
p-5\cdot d^2 \ra g,
$$
where 
$$
d=\begin{cases}  1 \qquad &\text{if $p \equiv 1,3 \Mod 8$,  $p \equiv 7\Mod {16}$, \text{or} \ $p \equiv 21\Mod {64}$}, \\
                           5 \qquad &\text{if $p \equiv 13, 45  \Mod {64}$,}\\
                           3 \qquad &\text{otherwise.}
                           \end{cases}
$$
This completes the proof. 
\end{proof}


\begin{thm}\label{thm235alpha}
For an integer $\alpha$ that is not divisible by $7$ with $5\leq \alpha\leq 43$,  the quaternary quadratic form $f=\df{2,3,5,\alpha}$ listed in \eqref{candi} is prime-universal. 
\end{thm}

\begin{proof}
Note that the class number of $g=\df{2,3,5}$ is $2$, and its genus mate is $g_1=\df{1,1,30}$.
One may easily check that
$$
g_1\prec_{7,0}g,\quad g_1\prec_{7,3}g,\quad g_1\prec_{7,5}g,\quad\text{and}\quad g_1\prec_{7,6}g.
$$
Therefore, for any $a\in\{0,3,5,6\}$, $Q(\text{gen}(g))\cap S_{7,a} \subset Q(g)$.
Consequently, $g$ represents all nonnegative integers that are not divisible by $3$, and are congruent to $0,3,5$, or $6$ modulo $7$.

Since one may directly check that $p \ra f$ for any prime $p$ such that $p\le 36\alpha$, we assume that $p> 36\alpha$.
If $p$ is congruent to $3,5$, or $6$ modulo $7$, then $p$ is represented by $g$.
Assume that $p$ is congruent to $1,2$, or $4$ modulo $7$. 
Then, one may easily check that
$$
p-\alpha\cdot3^2\cdot d^2\ra g,
$$
where
$$
d=\begin{cases} 1 \qquad &\text{if $p-2\alpha\equiv0,3,5$, or $6 \Mod 7$},\\
2 \qquad &\text{otherwise}.
\end{cases} 
$$
This completes the proof.   \end{proof}


\begin{thm} \label{thm23514} The quaternary quadratic form $f=\df{2,3,5,14}$ is prime-universal.
\end{thm}

\begin{proof}
Note that the class number of $g=\df{2,3,14}$ is two and its genus mate is
$$
g_1=\begin{pmatrix} 1&0&0\\0&10&4\\0&4&10\end{pmatrix}.
$$
One may easily check that
$$
g_1\prec_{8,3}g,\quad g_1\prec_{32,12}g\quad \text{and}\quad g_1\prec_{32,16}g.
$$
Thus if we put
$$
A=\left\{ m\in \n : m\equiv 3\Mod 8 \ \ \text{or}\ \ m\equiv 12\Mod {32}\ \ \text{or}\ \ m\equiv 16\Mod {32} \right\},
$$
then we have $Q(\gen(g))\cap A\subset Q(g)$.
Consequently, $g$ represents all integers in $A$ except for those of the form $3^{2s+1}(3t+2)$ for some nonnegative integers $s$ and $t$.
One may check directly that $p\ra g$ for any prime $p<2207$. Thus we assume that $p$ is a prime greater than or equal to  $2207$.
If we define
$$
d=\begin{cases} 0&\ \text{if}\ p\equiv 3\Mod 8,\\
                           1&\ \text{if}\ p\equiv 25\ \text{or}\ 29\Mod {32},\\
                           2&\ \text{if}\ p\equiv 7\Mod 8,\\
                           3&\ \text{if}\ p\equiv 1\ \text{or}\ 5\Mod {32},\\
                           5&\ \text{if}\ p\equiv 17\ \text{or}\ 21\Mod {32},\\
                           7&\ \text{if}\ p\equiv 9\ \text{or}\ 13\Mod {32},
                           \end{cases}
$$
then one may easily check that 
$$
n=p-5\cdot 3^2\cdot d^2\in A \quad \text{and} \quad n \not \equiv 0 \Mod 3.
$$
Thus we have $n\ra g$, which completes the proof.
\end{proof}


\begin{thm} \label{thm23521} The quaternary quadratic form $f=\df{2,3,5,21}$ is prime-universal. 
\end{thm}

\begin{proof}
One may check directly that $f$ represents  all primes less than $88211$.
Thus we assume  that $p$ is a prime greater than or equal to $88211$.
Note that the class number of $g=\df{3,5,21}$ is three and its genus mates are
$$
g_1=\begin{pmatrix} 5&0&0\\0&6&3\\0&3&12\end{pmatrix} \ \ \text{and}\ \ g_2=\begin{pmatrix} 2&1&0\\1&8&0\\0&0&21\end{pmatrix}.
$$
One may easily show that any positive integer $m$ which is congruent to 2 modulo 3 and relatively prime to 35 is represented by the genus of $g$.
If we define
$$
d=\begin{cases} 2&\ \text{if}\ p\equiv 1\Mod {12},\\
                           6&\ \text{if}\ p\equiv 5\Mod {12},\\
                           1&\ \text{if}\ p\equiv 7\Mod {12},\\
                           3&\ \text{if}\ p\equiv 11\Mod {12}, \end{cases}
$$
and 
$$
n=p-2\cdot 35^2\cdot d^2,
$$
then one may easily check that  $n$ is a positive integer congruent to $5$ modulo $12$ and is relatively prime to $35$. 
Thus we have
$$
n\equiv 1\Mod 4\quad  \text{and}  \quad n\ra \gen(g).
$$
Hence the integer $n$ is represented by a quadratic form in the genus of $g$. 
Suppose that $n\ra g_1$. Then there is a vector $(x,y,z)\in \z^3$ such that
$$
n=5x^2+6y^2+12z^2+6yz.
$$
Since $n\equiv 1\Mod 4$, we have either $y\equiv 0\Mod 2$ or $yz\equiv 1\Mod 2$.
If $y=2y'$ for some $y'\in \z$, then we have
$$
\begin{array}{rl}
n\!\!\!\!\!&=5x^2+24y'^2+12z^2+12y'z \\[0.4em]
&=3(y'+2z)^2+5x^2+21y'^2
\end{array}
$$
which means that $n\ra g$.
If $yz\equiv 1\Mod 2$, then we may write $y=z-2y'$.
Then we have
$$
\begin{array}{rl}
n\!\!\!\!\!&=5x^2+6(z-2y')^2+12z^2+6(z-2y')z \\[0.4em]
&=5x^2+24y'^2+24z^2-36y'z \\[0.4em]
&=3(y'+z)^2+5y^2+21(y'-z)^2
\end{array}
$$
and thus $n\ra g$.
Similarly, one may show that the supposition $n\ra g_2$ also implies that $n\ra g$.
This completes the proof. 
\end{proof}


\begin{thm} \label{thm23535} The quaternary quadratic form $f=\df{2,3,5,35}$ is prime-universal.
\end{thm}

\begin{proof}
One may directly check that $f$ represents all primes less than $596243$.
Thus we assume that $p$ is a prime greater than or equal to $596243$.
Note that the class number of $g$ is three and its genus mates are 
$$
g_1= \begin{pmatrix} 3&0&0\\0&10&5\\0&5&20\end{pmatrix}
\ \ \text{and}\ \ g_2=\begin{pmatrix} 2&1&0\\1&8&0\\0&0&35\end{pmatrix}.
$$
One may easily show that any positive integer $m$ which is congruent to 2 or 3 modulo 5, and is relatively prime to 21 is represented by the genus of $g$.
If we define
$$
d=\begin{cases} 13&\ \text{if}\ p\equiv 1\Mod {20},\\
                           10&\ \text{if}\ p\equiv 3\ \text{or}\ 7\Mod {20},\\
                           1&\ \text{if}\ p\equiv 9\Mod {20},\\
                           2&\ \text{if}\ p\equiv 11\Mod {20},\\
                           5&\ \text{if}\ p\equiv 13\ \text{or}\ 17\Mod {20},\\
                           26&\ \text{if}\ p\equiv 19\Mod {20},\\
                           \end{cases}  \quad \text{and} \quad n=p-2\cdot 21^2\cdot d^2,
$$
 then one may easily check that  
$$
n\equiv 3\Mod 4,\ \ n\ra \gen(g).
$$
For each $i=1,2$, one may easily show that any positive integer $m$ represented by $g_i$ which is congruent to 3 modulo 4 is also represented by $g$ as in the proof of Theorem \ref{thm23521}.
Therefore we have $n\ra g$, which implies that $p\ra f$.
\end{proof}

\section{Higher dimensional prime-universal diagonal quadratic forms}

 A prime-universal diagonal quadratic form  is called {\it proper} if there does not exist a prime-universal diagonal subform  whose rank is less than the rank of it.  Note that to find all prime-universal diagonal quadratic forms, it suffices to find all proper ones.  In this section, we classify all proper prime-universal diagonal quadratic forms with ranks greater than $4$. 

If a diagonal quadratic form $f$ fails to represent a prime, we call the smallest prime not represented by $f$ its {\em prime truant}.

According to the results in \cite{dw}, there are no prime-universal diagonal binary quadratic forms, and there are five proper prime-universal diagonal ternary quadratic forms, namely,
$$
\df{1,1,2}, \quad \df{1,1,3}, \quad \df{1,2,3}, \quad \df{1,2,4}, \quad \text{and} \quad \df{1,2,5}.
$$
Moreover, as proved in the previous section,  all prime-universal diagonal quaternary quadratic forms with minimum greater than $1$ are listed in Table 1 of \cite{dw}.
Therefore, it suffices to classify all proper prime-universal diagonal quadratic forms with rank greater than or equal to $5$, which are, in fact, listed in Table \ref{HigherCandidate}.

\begin{table}[ht]
\caption{Proper prime-universal forms $\langle a_1,a_2,\dots,a_k\rangle \ \ (5\le k\le6)$} 
\label{HigherCandidate}
	\begin{center}
		\begin{tabular}{ccccccc}
			\Xhline{1pt}
			$a_1$&$a_2$&$a_3$&$a_4$&$a_5$&$a_6$& Conditions on $a_k  \ (5\le k\le6)$ \\
			\hline
			
			$2$&$2$&$2$&$2$&$3$&& \\
			$2$&$2$&$2$&$3$&$a_5$&& $a_5 = 8,11,17$\\
			
			\hline
			
			$2$&$2$&$3$&$8$&$17$&& \\
			$2$&$2$&$3$&$11$&$17$&& \\
			$2$&$2$&$3$&$16$&$17$&& \\
			$2$&$2$&$3$&$17$&$a_5$&& $17\le a_5 \le 41$, $a_5\neq 26,32,35,40$\\
			
			\hline
			
			$2$&$3$&$3$&$3$&$4$&& \\
			$2$&$3$&$3$&$4$&$a_5$&& $a_5 = 4,6,10,13$\\
			
			\hline
			$2$&$3$&$4$&$4$&$a_5$&& $4\le a_5 \le 17$, $a_5\neq 5,8,9,11,16$\\
			$2$&$3$&$4$&$6$&$a_5$&& $6\le a_5 \le 23$, $a_5\neq 8,9,11,22$\\
			$2$&$3$&$4$&$7$&$a_5$&& $7 \le a_5 \le 17$, $a_5\neq 8,9,11,16$ \\
			$2$&$3$&$4$&$10$&$a_5$&& $10\le a_5 \le 23$, $a_5\neq 11,22$\\
			$2$&$3$&$4$&$12$&$13$&& \\
			$2$&$3$&$4$&$13$&$a_5$&& $13 \le a_5 \le 23$, $a_5\neq 13,22$\\
			
			\hline	
			
			$2$&$3$&$5$&$7$&$a_5$&& $a_5=7,19,28,34$ \\
			$2$&$3$&$5$&$19$&$19$&&  \\
			
			\hline
			$2$&$3$&$6$&$6$&$7$&&  \\			
			$2$&$3$&$6$&$7$&$a_5$&& $7\le a_5 \le 23$, $a_5\neq 8,19,20,22$  \\			
			$2$&$3$&$6$&$7$&$19$&$20$& \\
			$2$&$3$&$6$&$7$&$20$&$a_6$& $20 \le a_6 \le 67$, $a_6\neq 21,23,63,66$ \\
			
			\hline
			
			$2$&$3$&$7$&$7$&$a_5$&& $7\le a_5 \le 13$, $a_5\neq 7,8,9,10,12$ \\
			$2$&$3$&$7$&$7$&$7$&$10$&  \\			
			$2$&$3$&$7$&$9$&$a_5$&& $9\le a_5 \le 13$, $a_5\neq 9,12$ \\		
			$2$&$3$&$7$&$10$&$a_5$&& $10\le a_5 \le 23$, $a_5\neq 17,19,22$ \\
			$2$&$3$&$7$&$11$&$a_5$&& $11\le a_5 \le 17$, $a_5\neq 11,13,16$ \\
			$2$&$3$&$7$&$12$&$13$&& \\
			$2$&$3$&$7$&$13$&$a_5$&& $13\le a_5 \le 17$, $a_5\neq 13,16$ \\
			\Xhline{1pt}
		\end{tabular}
	\end{center}
\end{table}

\begin{thm}\label{HigherproperPU}
	Every diagonal quadratic form in Table \ref{HigherCandidate} is, in fact, proper prime-universal.
\end{thm}

We will prove Theorem \ref{HigherproperPU} at the end of this section.
Assuming that Theorem \ref{HigherproperPU} is proved, we first prove the following theorem, which guarantees that Table \ref{HigherCandidate} is the complete list of proper prime-universal diagonal quadratic forms of rank greater than or equal to $5$.

\begin{thm}\label{ThmHigherescalation}
	Let $f$ be a proper prime-universal diagonal quadratic form with rank greater than or equal to $5$. Then $f$ is isometric to one of the quadratic forms listed in Table \ref{HigherCandidate}.
\end{thm}
\begin{proof}
	Let  $f=\langle a_1,a_2,\cdots,a_n \rangle$ be a proper prime-universal diagonal quadratic form, where $a_i$'s are all integers such that $1\le a_1\le \dots \le  a_n$. 
	 For the sake of convenience, we define $f_i=\df{a_1,\ldots,a_i}$  for each $i$ with $1\le i\le n$.
Since we are assuming that $f$ is proper, $f_4$ should not be prime-universal.
	Following the argument used in the proof of Theorem 1.3 of \cite{dw}, one may easily obtain all candidates of $f_4$.
	After removing those forms that are already prime-universal among all possible candidates,  $f_4$ is isometric to one of the followings:
	$$
	\begin{array}{llll}
	\text{\bf (Case 1)}& \df{2,2,2,a_4} && \text{where  $a_4=2,3$,}\\
	\text{\bf (Case 2)}&\df{2,2,3,a_4} && \text{where  $a_4=8,11,16,17$,}\\
	\text{\bf (Case 3)}&\df{2,3,3,a_4} && \text{where  $a_4=3,4,6$,}\\
	\text{\bf (Case 4)}&\df{2,3,4,a_4} && \text{where  $a_4=4,6,7,10,12,13$,}\\
	\text{\bf (Case 5)}&\df{2,3,5,a_4} && \text{where  $a_4=7,19,28,34,37,39,42$,}
	\end{array}
	$$
	and
	$$
	\hspace{-0.65cm}\begin{array}{llll}
	\text{\bf (Case 6)}&\df{2,3,6,a_4} && \text{where  $a_4=6,7$,}\\	
	\text{\bf (Case 7)}&\df{2,3,7,a_4} && \text{where  $a_4=7,9,10,11,12,13$.}	
	\end{array}
	$$
	
Note that any diagonal form $\df{a_1,a_2,a_3,a_4}$ listed in the above is not prime-universal. 
	Although what we have to check is different case by case, we repeat the following process, starting from $i=5$:
	
	Let $p$ be the prime truant of $f_{i-1}=\df{a_1,\ldots,a_{i-1}}$. Since $f$ is prime-universal, it should represent $p$. Hence we have $a_{i-1}\le a_i \le p$.  
	For each integer $a$ with $a_{i-1}\le a\le p$,  it is possible that
	\begin{enumerate}[label={\rm(\arabic*)}]
		\item [(I)] The form $g_i=g_{i,a}:=\df{a_1,\ldots,a_{i-1},a}$ is prime-universal, and either
		\begin{enumerate}
			\item[(I-1)] $g_i$ is proper, or
			\item[(I-2)] $g_i$ is not proper.
		\end{enumerate}
		\item [(II)] The form $g_i$ is not prime-universal, and either
		\begin{enumerate}
			\item[(II-1)] $g_i$ fails to represent another prime $p'$ greater than $p$, or
			\item[(II-2)] $g_i$ still fails to represent the prime $p$.
		\end{enumerate}
	\end{enumerate}
	Note that any prime-universal form  contain a proper prime-universal form as its subform, one may use Table \ref{HigherCandidate} given above and Table 1 of \cite{dw} to check the prime-universality of $g_i$ as well as its properness.
	
	If (I-1) happens, then one may easily check that $g_i$ is listed in Table  \ref{HigherCandidate} and hence the theorem is proved. If (I-2) happens, then one may easily find a proper subform of $g_i$ which is listed in Table  \ref{HigherCandidate}.

	If (II) happens, we increase $i$ by $1$ and repeat the process.
	In particular, if (II-2) happens,  then one might come across the following situation: 
	any diagonal quadratic form $\df{a_1,\ldots,a_{i-1},a,\ldots,a}$ fails to represent $p$, no matter how many times the integer $a$ is repeated.
	Even if it happens, the equation 
	$$
	a_1x_1^2+\dots+a_{i-1}x_{i-1}^2+a_{i}x_{i}^2+\dots+a_nx_n^2=p
	$$
	has an integer solution, for $f$ represents the prime $p$. Therefore there is an integer $b$ with $a<b \le p$ such that $\df{a_1,\dots,a_i,a,b}$ represents $p$.  Now, we repeat the process from the diagonal quadratic form  
	$$
	\df{a_1,\dots,a_i,\overbrace{a,\dots,a}^{\text{$s$-times}},b}
	$$
	 for each $s\ge 1$.  From (Case 1) to (Case 7) given above, one may check that this process terminates in a finite step. Since the proofs of all cases are quite similar to each other, we provide the proofs of (Case 2) and (Case 6).
\vskip 0.3cm

	\noindent {\bf (Case 2)} $f_4=\df{2,2,3,a_4}$ with $a_4=8,11,16,17$.
	
	First, we consider the case when $f_4=\df{2,2,3,8}$.
	Since the prime truant of $f_4$ is $17$, we have $8\le a_5 \le 17$.
	For each $a=9,10,12,13,14,15$, the form $\df{2,2,3,a}$ is already prime-universal, and hence $f$ is not proper if $a_i=a$ for some $i\ge 5$.
	Therefore, we may assume that 
	$$
	a_i \in \{8,11,16,17\} \quad \text{or} \quad a_i \ge 18
	$$ 
	for any $i\ge 5$. If $a_i \ne 17$ for any $i \ge 5$, then the diagonal quadratic form $\df{2,2,3,8,a_5\dots,a_n}$ does not represent $17$. Therefore $a_i=17$ for some $i$. 
	Since $\df{2,2,3,8,17}$ is prime-universal, the only proper prime-universal diagonal quadratic form $f$ with $f_4=\df{2,2,3,8}$ is $\df{2,2,3,8,17}$.
	One may similarly prove the case when $a_4=11,16$.
	
	Next, we consider the case when $f_4=\df{2,2,3,17}$. Note that the prime truant of $\df{2,2,3,17}$ is $41$, and hence we have $17\le a_5 \le 41$.
	Since $\df{2,2,3,17}$ does not represent $1,6,9,15,41$, $f$ represents $41$ if and only if 
$$
a_i\in \mathcal A= \{17\le a \le 41 : a\neq 26,32,35,40\}
$$
	 for some $i\ge 5$. Furthermore, one may easily check by using Table \ref{HigherCandidate} that any quadratic form $\df{2,2,3,17,a_5}$ with $a_5 \in \mathcal A$ is proper prime-universal. 
\vskip 0.3cm

	\noindent {\bf (Case 6)} $f_4=\df{2,3,6,a_4}$ with $a_4=6,7$.   
	
	First, we consider the case when $f_4=\df{2,3,6,6}$.
	Note that the prime truant of $f_4$ is $7$, and $f$ can represent $7$ if and only if $a_i=7$ for some $i\ge5$. Also, note that $\df{2,3,6,6,7}$ is prime-universal, whereas $\df{2,3,6,7}$ does not represent $23$. 
	
	Next, we consider the case when $f_4=\df{2,3,6,7}$.
	Since the prime truant of $f_4$ is $23$, we have $7\le a_5\le 23$.
	Note that every form $\df{2,3,6,7,a_5}$ with $7\le a_5\le 23$ and $a_5\neq 19,20,22$ is proper prime-universal.
	On the other hand, the prime truants of $\df{2,3,6,7,19}$, $\df{2,3,6,7,20}$, and $\df{2,3,6,7,22}$ are $23$, $67$, and $23$, respectively.
	
	Now, assume that $a_5=19$. If $a_i\in \{21,23\}$ for some $i\ge 6$, then $f$ is prime-universal but not proper. Since $\df{2,3,6,7,19}$ does not represent integers $1,4,23$, $f$ can represent $23$ if and only if $a_i=20$ for some $i\ge 6$. Note that $\df{2,3,6,7,19,20}$ is proper prime-universal. 
	Therefore, the only proper prime-universal diagonal quadratic form $f$ with $f_5=\df{2,3,6,7,19}$ is $\df{2,3,6,7,19,20}$.
	
	Next, assume that $a_5=20$. Since the prime truant of $\df{2,3,6,7,20}$ is $67$, we have $20\le a_6 \le 67$. Note that $\df{2,3,6,7,20,a_6}$ with  $20\le a_6 \le 67$ and $a_6\neq 63,66$ is proper  prime-universal, and $f=\df{2,3,6,7,20,\{63 \text{ or } 66\},a_7,\cdots}$ can represent $67$ if and only if $a_i=64,65,67$ or $a_i=67$, respectively, for some $i\ge 7$. However, since $\df{2,3,6,7,20,a}$ with $a=64,65,$ or $67$ is already prime-universal, any prime-universal form $\df{2,3,6,7,20,\{63 \text{ or }66\},a_7,\cdots}$ is not proper.
	
	Finally, assume that $a_5=22$.  Since $f_5=\df{2,3,6,7,22}$ does not represent $23$, $a_i=23$ for some $i\ge 6$. Note that $\df{2,3,6,7,23}$ is proper prime-universal. 
\end{proof}

\begin{thm}\label{thm222}
The quadratic form $f=\df{2,2,2,3}$ represents all primes but not $17$.
Hence, the quadratic forms $\df{2,2,2,2,3}$, and $\df{2,2,2,3,a_5}$ with $3 \le a_5 \le 17$ and $a_5\neq 16$ are prime-universal.
\end{thm}

\begin{proof} 
Note that the class number of $g=\df{2,2,2}$ is $1$. 
Therefore $Q(\gen(g))=Q(g)$, and consequently $g$ represents all nonnegative even integers except for those of the form $4^s(16t+14)$ for some nonnegative integers $s$ and $t$.

Since one may directly check that $p\ra f$ for any prime $p$ such that $p\le 73$ except for $17$, we assume that $p\ge 79$.
Then, one may easily check that
$$
p-3\cdot d^2 \ra g,
$$
where
$$
d=
\begin{cases} 
5&\ \text{if}\ p\equiv 27 \Mod {32},\\
3&\ \text{if}\ p\equiv 1,3,17 \Mod {12},\\
1&\ \text{otherwise}.
\end{cases}
$$
Therefore, $p \ra f$.
Now, the proof of the last assertion is straight forward.
\end{proof}

\begin{thm}\label{thm223}
The quadratic form $f=\df{2,2,3,17}$ represents all primes but not $41$.
Hence, the quadratic forms $\df{2,2,3,a_4,17}$ with $a_4=8,11,16$, and $\df{2,2,3,17,a_5}$ with $17 \le a_5 \le 41$ and $a_5\neq 26,32,35,40$ are prime-universal.
\end{thm}

\begin{proof}
Note that the class number of $g=\df{2,2,3}$ is $1$. 
Therefore $Q(\gen(g))=Q(g)$, and consequently $g$ represents all nonnegative integers except for those of the forms $8t+1$ and $9^s(9t+6)$ for some nonnegative integers $s$ and $t$.

Since one may directly check that $p\ra f$ for any prime $p$ such that $p\le 607$ except for $41$, we assume that $p\ge 613$.
Then, one may easily check that
$$
p-17\cdot 6^2 \ra g.
$$
Therefore, $p \ra f$.
Now, the proof of the last assertion is straight forward.
\end{proof}

\begin{thm}\label{thm233}
The quadratic form $f=\df{2,3,3,4}$ represents all primes but not $13$.
Hence, the quadratic forms $\df{2,3,3,3,4}$ and $\df{2,3,3,4,a_5}$ with $4 \le a_5 \le 13$ and $a_5\neq 12$ are prime-universal.
\end{thm}

\begin{proof}
Note that the class number of $g=\df{2,3,3}$ is $1$. 
Therefore $Q(\gen(g))=Q(g)$, and consequently $g$ represents all nonnegative integers except for those of the form $9^s(3t+1)$ for some nonnegative integers $s$ and $t$.

Since one may directly check that $p\ra f$ for any prime $p$ such that $p\le 13$ except $13$, we assume that $p\ge 17$.
Then, one may easily check that
$$
p-4\cdot d^2 \ra g,
$$
where
$$
d=
\begin{cases} 
1&\ \text{if}\ p\equiv 1,7 \Mod {9},\\
2&\ \text{if}\ p\equiv 4 \Mod {9}.\\
\end{cases}
$$
Therefore, $p \ra f$.
Now, the proof of the last assertion is straight forward.	
\end{proof}

\begin{thm}\label{thm234}
The quadratic forms
$$
\begin{array}{llll}
\df{2,3,4,4},& \df{2,3,4,7} &&\text{represents all primes but not } 17,\\
\df{2,3,4,6},& \df{2,3,4,10},& \df{2,3,4,13}&\text{represents all primes but not } 23,\\
\df{2,3,4,12}&&&\text{represents all primes but not } 13.\\
\end{array}
$$
Hence, all quadratic forms of the form $\df{2,3,4,a_4,a_5}$ listed in Table \ref{HigherCandidate} are prime-universal.
\end{thm}

\begin{proof}
Let $f$ be one of the six quaternary forms given above.
Let 
$$
\alpha=
\begin{cases} 
7&\ \text{if}\ f=\df{2,3,4,7},\\
13&\ \text{if}\ f=\df{2,3,4,13},\\
3&\ \text{otherwise},\\
\end{cases}
$$
and let $g$ be the ternary sub-form such that $f=g \perp \df{\alpha}$. Note that 
$g$ represents all even integers except for those of the form
$$
\begin{cases}{}
4^s(16t+10)						&\text{if}\ g=\df{2,3,4} \text{ or } \df{2,4,12}, \\
4^s(16t+14)						&\text{if}\ g=\df{2,4,4},\\
4^s(32t+20) 					&\text{if}\ g=\df{2,4,6},\\
25^s(50t+\{20 \text{ or } 30\}) &\text{if}\ g=\df{2,4,10}.
\end{cases}
$$
For the proof of the case when $g=\df{2,3,4}$, see the proof of Theorem \ref{start}, and for all the other cases, note that the class number of $g$ is $1$.

Since one may directly check that $p\ra f$ for any prime $p$ such that $p\le 317$ except for the only one prime given above, we assume that $p\ge 331$.
Then one may easily check that
$$
p-\alpha\cdot d^2 \ra g,
$$
where
$$
d=\begin{cases}
5 & \text{if } \ell=\df{2,4,10},\\
1,3, \text{ or }5 & \text{otherwise, defined according as the residue of $p$ modulo $32$},
\end{cases}
$$
Therefore, $p \ra f$.
Now, the proof of the last assertion is straight forward.		
\end{proof}

\begin{thm}\label{thm236}
The quadratic form $f=\langle 2,3,6,7\rangle$ represents all primes but not $23, 47$, and $67$. 
Hence, the quadratic forms 
$$
\df{2,3,6,6,7}, \ \df{2,3,6,7,a_5},  \quad \text{and} \quad  \df{2,3,6,7,a_5,a_6}
$$
 listed in Table \ref{HigherCandidate} are prime-universal.
Moreover, the quadratic forms 
$$
\begin{array}{lll}
\df{2,3,6,7,19}, & \df{2,3,6,7,22} &\text{ represents all primes but not } 23,\\
\df{2,3,6,7,20} && \text{ represents all primes but not } 67.
\end{array}
$$
\end{thm}
\begin{proof}
Note that the class number of $g=\df{2,3,6}$ is $1$. 
Thus $Q(\gen(g))=Q(g)$, and consequently $g$ represents all nonnegative integers except for those of the forms $4^s(8t+7)$ and $9^s(3t+1)$ for some nonnegative integers $s$ and $t$.

Since one may directly check that $p\ra f$ for any prime $p$ such that $p\le 251$ except for $23,47,$ and $67$, we assume that $p\ge 257$.
Then, one may easily check that
$$
p-7\cdot d^2 \ra g,
$$
where
$$
d=\begin{cases}  
0&\ \text{if}\ p\equiv 5,11,17 \Mod {24},\\
2&\ \text{if}\ p\equiv 1, 7, 13 \Mod {24},\\
4&\ \text{if}\ p\equiv 19 \Mod {24},\\
6&\ \text{if}\ p\equiv 23 \Mod {24}.\\
\end{cases}
$$
Therefore, $p \ra f$.
Now, the proof of the last assertion is straight forward.	
%
\end{proof}

\begin{thm}\label{thm235-237}
The quadratic form $\df{2,3,5,19}$ represents all primes but not $43$. Hence, the quadratic form $\df{2,3,5,19,19}$ is prime-universal.
Moreover, the quadratic forms 
$\df{2,3,5,7,a_5}$, $\df{2,3,7,a_4,a_5}$, and $\df{2,3,7,7,7,10}$ listed in Table \ref{HigherCandidate} are prime-universal.	
\end{thm}

\begin{proof}
Since $19$ is not divisible by $7$, $\df{2,3,5,19}$ represents every prime $p\ge36\times 19=684$ (see the proof of Theorem \ref{thm235alpha}). One may easily check that $\df{2,3,5,19}$ represents all primes less than $684$ except  for $43$. Also, since $\df{2,3,5,19,19}$ represents $43$, it is prime-universal.

Now, we prove ``Moreover part" of the theorem. Let $f=\df{2,3,7,\alpha,\beta}$ be one of the quadratic forms that we should prove it is prime-universal. (for the quadratic form $\df{2,3,5,7,a_5}$, we let $\alpha=5$ and $\beta=a_5$, and we keep the senary quadratic form $\df{2,3,7,7,7,10}$ in mind for a moment).
Since one may easily check that $p\ra f$ for any prime $p$ such that $p\le 13121$, we assume that $p\ge 13127$.
Note that the class number of $g=\df{2,3,7}$ is $3$ and all isometric classes in the genus are
$$
g=\df{2,3,7},\ \ g_1=\df{1,3,14} \ \ \text{and}\ \ 
g_2= \begin{pmatrix} 2&1&1\\1&3&1\\ 1&1&9\end{pmatrix}.
$$
Now, one may easily check that 
$$
g_1\prec_{3,0}g, \ \ \text{and} \ \ g_2\prec_{3,0}g.
$$
Therefore, $Q(\gen (g)) \cap 3\z \subset Q(g)$, and consequently $g$ represents all nonnegative odd integers that are multiple of $3$.
In the case when both $\alpha$ and $\beta$ are not divisible by $3$, then one may easily check that 
$$
p-\alpha\cdot d^2-\beta \cdot e^2 \ra g
$$
for some integers $d,e\in\{0,1,2\}$.
Therefore, $p \ra f=\df{2,3,7,\alpha,\beta}$ in this case, and we may also have $p\ra \df{2,3,7,7,7,10}$.

Now, the remaining quadratic forms $f$ under consideration are
$$
\begin{array}{llll}
\df{2,3,7,9,\beta} &\text{with } \beta=11,13, & \df{2,3,7,10,\beta} &\text{with } \beta=12,15,18,21,\\
\df{2,3,7,11,\beta} &\text{with } \beta=12,15,&\df{2,3,7,12,13}, &  \df{2,3,7,13,15}.\\
\end{array}
$$
One may check by a computer program that
$$
g_1\prec_{30,r}g, \ \ \text{and} \ \ g_2\prec_{30,r}g
$$
for any integer $r$ in $R_{30}=\{0,2,3,6,8,9,10,12,15,18,20,21,22,24,27,28\}$.
Therefore, $Q(\gen(g))\cap \{m : m\equiv r \Mod{30} \text{ for some }r\in R_{30}\}\subset g$, and consequently $g$ represents all nonnegative integers in the set
$$
M_{30}=\{m \mid m \not \equiv 0,6,22,24 \Mod{32}, \ m \equiv r \Mod{30} \text{ for some } r\in R_{30}\}.
$$
Using this, for each $\gamma=9,11,13,21$, one may check that 
$$
p-\gamma\cdot d^2 \ra g
$$
for some $0\le d \le 25$, defined according as the residue of $p$ modulo $32\cdot 15$.
Therefore, $p \ra \df{2,3,7,\gamma}$.
This proves the prime-universality of $f$ except for those with $\alpha=10$ and $\beta=12,15,18$, namely,
$$
\df{2,3,7,10,12},\ \ \df{2,3,7,10,15},\ \ \text{and} \ \ \df{2,3,7,10,18}.
$$
To prove in these cases, one may check by a computer program that 
$$
g_1\prec_{42,r}g, \ \ \text{and} \ \ g_2\prec_{42,r}g
$$
for any $r$ in $R_{42}=\{0, 3, 6, 7, 9, 12, 14, 15, 18, 21, 24, 27, 28, 30, 33, 35, 36, 39\}$. Therefore, $Q(\gen(g))\cap \{m : m\equiv r \Mod{42} \text{ for some }r\in R_{42}\}\subset g$, and consequently $g$ represents all nonnegative integers in the set
$$
M_{42}=\{m \mid m \not \equiv 0,6,22,24 \Mod{32}, \ m \equiv r \Mod{42} \text{ for some } r\in R_{42}\}.
$$
Using this, for any $\alpha=10$ and $\beta=12,15,18$,
one may check that 
$$
n=p-\alpha\cdot d^2 -\beta\cdot e^2 \ra g
$$
for some integers $d\in\{0,1,2\}$ and $0\le e \le 6$, defined according as the residue of $p$ modulo $32\cdot 21$.
Therefore, $n\ra f$ in this case.
This completes the proof.
\end{proof}

\begin{rmk}\label{rmk235-237}
From the proof of Theorem \ref{thm235-237}, one may verify that
$$
\begin{array}{lll}
\df{2,3,7,9} 		&\text{represents all primes but not } 13,97,\\
\df{2,3,7,11}	 	&\text{represents all primes but not } 17,\\
\df{2,3,7,13} 	&\text{represents all primes but not } 17,\\
\df{2,3,10,21} 	&\text{represents all primes but not } 13,17,43,47.\\
\end{array}
$$
\end{rmk}

 Now, we are ready to prove Theorem \ref{HigherproperPU}.

\begin{proof}[Proof of Theorem \ref{HigherproperPU}]
The properness can be checked through a straight forward calculation.
The prime-universality of each diagonal quadratic form in Table \ref{HigherCandidate} follows from Theorems \ref{thm222}-\ref{thm235-237}.
\end{proof}

\begin{cor}
Let $f$ be a diagonal quadratic form and let
$$
\mathcal{S}:=\{2,3,5,7,13,17,23,41,43,67\}.
$$
The diagonal quadratic form $f$ represents every prime in $\mathcal{S}$  if and only if $f$ is prime-universal.
Moreover, the set $S$ is minimal in the sense that any proper subset of $\mathcal{S}$ does not guarantee the prime-universality. 
\end{cor}

\begin{proof}
Note that if we repeat the proof of Theorem \ref{ThmHigherescalation} under the assumption that $f$ represents all  primes in $S$, we arrive at the same conclusion.
Also, note that for each prime $p \in S$, some quadratic forms representing all primes except $p$ are given in Table \ref{exceptOne}.
The representabilities of primes for each of the first four quadratic forms $\df{1,3,4}$, $\df{1,1,6}$, $\df{1,2,6,10}$, and $\df{1,1,1,9}$ in Table \ref{exceptOne} can be proved as usual, since the class number of each of the forms $\df{1,3,4}$, $\df{1,1,6}$, $\df{1,2,6}$, and $\df{1,1,1}$ is $1$.
Representabilities of primes for each of the other quadratic forms are proved in Theorems \ref{thm222}-\ref{thm235-237} and Remark \ref{rmk235-237}. Therefore, the set $S$ is minimal.
\end{proof}

\begin{table}[ht]
	\caption{Diagonal forms representing all primes but only one} \label{exceptOne}
	\begin{center}
		\begin{tabular}{c|c}
			\Xhline{1pt}
			$f$ & Exception of $f$\\ 
			\hline
$\df{1,3,4}$	&	 $2$\\

$\df{1,1,6}$	&	 $3$\\
\hline
$\df{1,2,6,10}$	&	 $5$\\

$\df{1,1,1,9}$	&	 $7$\\

$\df{2,3,3,4}$, $\df{2,3,4,12}$ &	$13$\\
\hline
$\df{2,2,2,3}$, $\df{2,3,4,4}$, $\df{2,3,4,7}$,	&	 \multirow{2}{*}{$17$}\\
$\df{2,3,7,11}$, $\df{2,3,7,13}$&\\
\hline
$\df{2,3,4,6}$, $\df{2,3,4,10}$, $\df{2,3,4,13}$, &  \multirow{2}{*}{$23$} \\ 
$\df{2,3,6,7,19}$, $\df{2,3,6,7,22}$	&  \\
\hline
$\df{2,2,3,17}$			&$41$\\

$\df{2,3,5,19}$			&$43$\\

$\df{2,3,6,7,20}$		&$67$\\			
\Xhline{1pt}
		\end{tabular}
	\end{center}
\end{table}

\begin{rmk} {\rm As pointed out by \cite{dw},  Bhargava asserted that if any quadratic form, which is not necessarily diagonal, represents
$$
2, \ 3, \ 5, \ 7, \ 11, \ 13, \ 17, \ 23, \ 29, \ 31, \ 37, \ 41, \ 43, \ 67, \ 73,
$$
then it is prime-universal.  However,  no proof of this has appeared in the literature to the author's knowledge.}  
\end{rmk}

\end{document}